\documentclass[a4paper,12pt]{amsart}
\usepackage{amsmath, amsfonts, amsthm, amssymb, mathrsfs, pb-diagram, graphics}
\usepackage{verbatim, color, tikz, fullpage}
\usepackage{enumerate}
\usepackage{geometry}

\newtheorem{thm}{Theorem}
\newtheorem{lem}{Lemma}
\newtheorem{cor}{Corollary}

\newcommand{\Z}{\mathbb{Z}}
\newcommand{\R}{\mathbb{R}}
\newcommand{\LL}{\mathbb{L}}

\newcommand{\eps}{\epsilon}
\newcommand{\om}{\omega}

\newcommand{\G}{\mathcal{G}}
\newcommand{\Ge}{\G_{\eps}}
\newcommand{\Geh}{\hat{\G}_{\eps}}

\newcommand{\GP}{\mathbf{G}_P}
\newcommand{\gT}{\mathbf{T}}
\newcommand{\cE}{\mathcal{E}}
\newcommand{\cR}{\mathcal{R}}
\newcommand{\One}{\mathbf{1}}

\DeclareMathOperator{\vol}{vol}
\DeclareMathOperator{\proj}{Proj}

\DeclareMathOperator{\interior}{int}
\DeclareMathOperator{\card}{card}

\usepackage{hyperref}
\hypersetup{colorlinks,%
citecolor=red,%
linkcolor=blue,%
}
\usepackage{color}

\begin{document}
	\title{Fourier transforms of polytopes, solid angle sums, and discrete volume}
	
\makeatletter
\@namedef{subjclassname@2010}{\textup{2010} Mathematics Subject Classification}
\makeatother
\subjclass[2010]{primary: 52C10, secondary: 52C15, 52C17, 32A27}
\keywords{lattice points, solid angle, Poisson summation, Fourier transform, polytope, Bernoulli polynomial, discrete volume, face poset}

\author{Ricardo Diaz}
\address{School of Mathematical Sciences, University of Northern Colorado,
Ross 2239,
Campus Box 122,
501 20th Street,
Greeley, CO 80639}
\email{ricardo.diaz@unco.edu}

\author{Quang-Nhat Le}
\thanks{The second author has received partial funding from the European Research Council (ERC) under the European Union's Horizon 2020 research and innovation programme (grant agreement No [ERC StG 716424 - CASe])}
\address{Department of Mathematics, Brown University, 
Box 1917, 151 Thayer Street, Providence, RI 02912 and
Einstein Institute of Mathematics, Hebrew University of Jerusalem, Jerusalem 9190401, Israel}
\email{qnhatle@math.brown.edu, qnhatle@math.huji.ac.il}

\author{Sinai Robins}
\address{Instituto de Matematica e Estatistica, Universidade de S\~ao Paulo, Rua do Matao 1010, 05508-090 S\~ao Paulo, Brazil}
\email{sinai\_robins@brown.edu}

\address{Department of Mathematics, Brown University, 
	Box 1917, 151 Thayer Street, Providence, RI 02912}
\email{sinai\_robins@brown.edu}

\begin{abstract}
		Given a real closed polytope $P$, we first describe the Fourier transform of its indicator function by using iterations of Stokes' theorem.  We then use the ensuing Fourier transform formulations, together with the Poisson summation formula, to give a new algorithm to count  fractionally-weighted lattice points inside  the one-parameter family of all real dilates of  $P$.   The combinatorics of the face poset of $P$ plays a central role in the  description of the Fourier transform of $P$. 
		We also obtain a closed form for the codimension-1 coefficient that appears in an expansion of this sum in powers of  the real dilation parameter $t$.  This closed form generalizes some known results about the Macdonald solid-angle polynomial, which is the analogous expression traditionally obtained by requiring that $t$ assumes only integer values.
 Although most of the present methodology applies to all real polytopes, a particularly nice application is to the study of all real dilates of integer (and rational) polytopes.  	 
\end{abstract}

\maketitle

\section{Introduction}

	 It is a classical problem in the geometry of numbers \cite{henkschurmannroots} to find ways of counting the number of lattice points in a general convex body.   The number of lattice points within the body can be regarded as a discrete analog of the volume of the body.   For closed polytopes in $\R^d$, this lattice-point problem  has attracted the attention of mathematicians working in a variety of fields.   Number theorists have applied lattice-point counting inside symmetric bodies in $\R^d$ to get bounds on norms of ideals \cite{siegel},  algebraic geometers have used properties of toric varieties to analyze this problem \cite{danilov}, \cite{fulton}, \cite{pommersheim}, statisticians have used lattice-point counting to help them enumerate contingency tables \cite{diaconisgangoli}, and combinatorialists have used lattice point enumeration to analyze the enumerative geometry of polytopes \cite{Barany}, \cite{BeckBraunKoppeSavageZafeirakopoulos}, \cite{McMullen1}, \cite{postnikovpermutahedra}, \cite{Stanleybook}.

If $P$ is an integer polytope, Ehrhart showed that the integer dilates of $P$ have the property that the number of integer points in $tP$ is a polynomial in $t$ (the Ehrhart polynomial). 
Ehrhart showed moreover that this  polynomial satisfies a reciprocity law \cite{Ehrhart1, Ehrhart2}, \cite{Stanleyreciprocity}.
The Ehrhart polynomial is easy to define, but unfortunately it fails to behave additively under natural geometrical operations, in particular the  gluing and subdividing of polytopes.  To ensure additivity, lattice points located at lower-dimensional boundary faces are assigned fractional weights, such as weight $\frac{1}{2}$ on the codimension-one facets of $P$, and even smaller dihedral-angle weights on the faces of codimension two, and so on.  Macdonald showed that if $P$ is an integer polytope and if $t$ is an integer, then the  sum of these fractionally-weighted lattice points inside the dilated polytope $tP$  depends polynomially on $t$ \cite{macdonald1, macdonald2}.   This solid-angle polynomial has the pleasing and computationally efficient property that it is additive on polyhedra with disjoint interiors; and is superior to the Ehrhart polynomial as an approximation to the  volume of  $tP$ (identity \eqref{solidanglesum2} below).   McMullen \cite{McMullen1} and Lawrence  \cite{lawrence} have developed a general theory of such additive valuations for polytopes from a combinatorial perspective.
	
	Here we explore solid-angle polytope addition formulas from the Fourier-analytic perspective. The use of Fourier analysis to solve discrete enumeration problems has a lengthy tradition extending from Siegel's classical approach to the Geometry of Numbers \cite{siegel}, to the work of Barvinok \cite{barvinok1, barvinok2, barvinokpommersheim}, Randol \cite{Randol1, Randol2}, Skriganov \cite{Skriganov1, Skriganov2, Skriganov3}, Beck  \cite{JozsefBeck}, and Brion and Vergne \cite{brionvergne}, and Diaz and Robins \cite{diazrobins}.  We give a more detailed account of some of the related historical development of lattice point enumeration, and some relations to solid angles, in section \ref{remarks}. 
		
	The solid angle sum, taken over all lattice points in a  polytope, will be computed by  applying Poisson summation to smoothings of the indicator function of the polytope.  This method converts the solid-angle enumeration problem to a new summation problem: summing a damped version of the  Fourier transform of the indicator function of $P$  over a dual lattice in the ($d$-dimensional) frequency domain.  These lattice sums require careful treatment because they diverge  when the damping factors are ignored. As an additional complication,  the Fourier transform has  anisotropic behavior: it decays at different rates in different directions in the frequency domain. For generic values of the wave-vector in the frequency domain, the Fourier transform is described rather simply by an exponential-rational function whose  rate of decay  depends only on the dimension of the polytope, but somewhat frustratingly, the lattice over which the Poisson sum will be  evaluated  contains non-generic points where the rate of decay is slower than the generic case.  
	
	An exact description of the Fourier transform will be found by a chain of iterated  applications of Stokes' theorem, and with each iteration we gain one additional unit rate of decay. This chain of  iterations halts prematurely precisely when the wave-vector is non-generic, and is orthogonal to one of the faces of $P$.  We will use the graded structure of  the Fourier transform of the indicator function  to establish the quasi-polynomiality  of the solid-angle sum for all real dilates of a rational polytope 
(see \eqref{quasipolynomial} for the classical definition of a quasipolynomial and Theorem   \ref{thm:Struc1} for our extension).

The normalized {\bf solid angle} fraction that a  convex $d$-dimensional polytope $P$ subtends  at any point $x \in \R^d$ is defined  by 
$\om_P(x)=\lim_{\eps\to\infty}{     \frac{\vol(S_{d-1}(x,\eps) \cap P)}{\vol(S_{d-1}(x,\eps))}}$. 
	It measures the fraction of a small $(d-1)$-dimensional sphere $S_{d-1}(x,\eps))$, centered at $x$, that intersects the polytope $P$.  
	\footnote{Note that balls and spheres can be used interchangeably in this definition - the fractional weight is the same using either method.}
	
	 It follows from the definition that 
	 $0 \leq \om_P(x) \leq 1$ for all $x \in \R^d$, $\om_P(x) = 0$ when $x \notin P$, and 
	$\om_P(x) = 1$ when $x \in \interior(P)$.  When $x$ lies on a codimension-two face of $P$, for example, then    $\om_P(x)$ is the fractional dihedral angle subtended by $P$ at $x$.	 Macdonald defined, for each positive integer $t$, the finite sum 
\begin{equation} \label{solidanglesum1}
A_P(t) :=  \sum_{x\in \Z^d} \om_{tP}(x),
\end{equation}
where $tP$ is the $t$'th dilation of the polytope $P$.  
Using purely combinatorial methods,  Macdonald  showed  that for any integer polytope $P$, and for {\bf positive integer values}  of $t$,
 \begin{equation} \label{solidanglesum2}
A_P(t) = (\vol P)  t^d + a_{d-2} t^{d-2} + a_{d-4} t^{d-4} + \cdots +   
    \begin{cases}
      a_1 t & \text{if } \dim P \text{ is odd},\\
      a_2 t^2  & \text{if } \dim P \text{ is even},
    \end{cases}
\end{equation} 
 hence of course it is a polynomial function of $t$ \cite{macdonald1, macdonald2}.  It is important to realize that in \eqref{solidanglesum2},  $t$ is restricted to integer values. 
 This discrete approximation to the volume for $tP$, also called the 
{\bf solid angle sum}, resembles the Ehrhart polynomial for $P$, the latter being defined by $|tP \cap \Z^d|$. There are some elementary and sometimes useful relations between the two polynomials, which are given in section \ref{remarks}.  

One interesting feature of the theory developed here is that our main result offers a geometric description for each of the quasi-coefficients $a_j(t)$.  For example, it turns out (from Theorem \ref{thm:main}) that the codimension-2 term $a_{d-2}$
 in 
\eqref{solidanglesum2} is described precisely by considering all the chains in the face poset of $P$ that terminate in a codimension-2 face of $P$, together with certain explicit $2$-dimensional lattice sums (see \eqref{complicatedcoeff}) that are performed on the orthogonal $2$-dimensional lattice to each codimension-2 face. 

 Our objective is to extend Macdonald's results to all real values of the dilation parameter $t$, and obtain some explicit  formulas for any coefficient of the solid angle sum $A_P(t)$ in this more general setting.  We should also remark that  our new formulations for the quasi-coefficients can sometimes be computationally intensive, depending on the geometry of $P$.

 We first outline an intuitive approach  that glosses over the technical details.  The Poisson Summation Formula asserts that, for any rapidly decreasing function $f$ on $\R^d$,
\begin{equation}
\sum_{m  \in\Z^d}  f(m)=\sum_{\xi \in\Z^d}\hat{f}(\xi),
\end{equation}
where  the Fourier transform of $f$ is defined by
\begin{equation}
\hat f(\xi) :=  \int_{\R^d}  f(x) e^{-2\pi i \langle \xi, x \rangle} dx,
\end{equation}
the integration being with respect to the uniform Lebegue measure on $\R^d$.
Now apply Poisson summation to smoothed versions of the indicator function of a polytope, taking  
$f:= \One_P \ast \Ge$,  with  
$\Ge :=   \eps^{-d/2}e^{-\pi\|x\|^2/\eps}$ the usual one-parameter family of dilated Gaussians:

\begin{align}\label{intuition1}
		\sum_{x\in\Z^d} (\One_P \ast \Ge)(x) &= \sum_{\xi\in\Z^d} \widehat{(\One_P \ast \Ge)}(\xi)    
		  \\ 
			&= \sum_{\xi\in\Z^d} \hat{\One}_P(\xi) \Geh(\xi) \\ \label{rhs}
			&= \sum_{\xi\in\Z^d} \hat{\One}_P(\xi) e^{-\pi\eps\|\xi\|^2}.
\end{align}

It is straightforward  to  analyze the sum on the left. Pointwise, we will see that $\lim_{\epsilon \rightarrow 0} (\One_P \ast \Ge)(x) =   \omega_P(x)$ 
(see section \ref{sec:Fourier}), and in fact 
as  $\epsilon \rightarrow 0$  the left-hand-side of  \eqref{intuition1}  converges to the solid angle sum  taken over all lattice points in the polytope $P$.  The right-hand-side of \eqref{rhs} will be evaluated by analyzing  the rather delicate structure of the Fourier transform $\hat \One_P(\xi)$, which henceforth will be abbreviated to $\hat P (\xi)$ for ease of reading.  

\medskip
{\bf Acknowledgement}.   The third author is grateful for the partial support of FAPESP grant Proc. 2103 / 03447-6, Brazil, and for the support of ICERM, at Brown University.  The second author would like to thank to Karim Adiprasito, whose grant \textit{ERC StG 716424 - CASe} supports part of this work. He is grateful for the support of ICERM, at Brown University, and would like to express his deepest gratitude to Richard E. Schwartz for his encouragement on this project.

\section {Statements of results}

Stokes' Theorem will be used repeatedly to  express $\hat P (\xi)$ as a weighted linear combination of Fourier transforms of its lower-dimensional faces. First we sketch the essential ideas. Applying Stokes' theorem to the   vector-field  $\vec \xi e^{- 2\pi i \langle x, \xi \rangle}$, we obtain

$$\int_{x\in P} - 2\pi i ||\xi||^2  e^{- 2\pi i \langle x, \xi \rangle} dx = \int_{\partial P} e^{- 2\pi i \langle x, \xi \rangle} \vec \xi \cdot n \ d\sigma$$ where $d \sigma$ represents Hausdorff measure on the bounding facets of $P$ and $\vec n$ represents the outward-pointing unit normal vector to each facet. That is, for all nonzero $\xi$, 
$ \hat P(\xi) = \frac{1}{(- 2\pi i)} \sum_{F\subset \partial P} \frac{ \vec \xi \cdot n_F}{  ||\xi ||^2} \hat F(\xi)$
with the understanding that the integral that defines each $\hat F$ is taken with respect to Hausdorff measure that matches the dimension of the facet $F \subset \partial P$. 

Now iterate this identity, by replacing $P$ by $F$ and restricting the linear form $\phi(x) =\langle x, \xi \rangle $ to the lower-dimensional facet $F$, which can be accomplished by removing  the component of $\vec \xi$ that is parallel to $n_F$.  If the vector $\xi$ is chosen generically, so that the linear function $\phi(x)= \langle x, \xi \rangle$ is never constant on any face of $P$ of positive dimension, this method can be iterated downward through all lower-dimensional faces  until one finally reaches the vertices of $P$, at which stage the final contribution in the computation chain is simply an exponential factor evaluated at each vertex. Thus the generic situation is that for a polytope of dimension $d$, $\hat P(\xi)$ is a sum of products, each product  comprising  $d$ algebraic factors that are  homogeneous of degree $(-1)$ and  an exponential factor  evaluated at a vertex of the polytope.  The exceptional non-generic cases occur precisely when the chain breaks prematurely, because $\phi(x)$ is constant on a face. When this occurs, however, no further integration is required on such a face, because the integrand in  the Fourier integral is constant.
The essential steps  in this iterative process will now be written down in  greater detail, using more precise notation. 

\begin{thm}[Combinatorial Stokes Formula for $\hat P$] \label{thm:DSF}
		Let $F$ be a polytope in $\R^d$  that  is not just a point: its dimension satisfies $1 \leq \dim(F) \leq n$.  Let $\phi_{\xi}(x)=-2\pi\langle \xi,x \rangle$ denote the real linear phase function in the integral formula of the definition of $\hat{F}$. We denote by $\proj_{F} (\xi)$ the orthogonal projection of $\xi$, considered as a vector in $\mathbb R^d$, onto  $F$.   For each (codimension-one) facet $G\subset F$ let  $N_F(G)$ be the  unit normal vector to  $G$  that points out of $F$.  Then 
\medskip
\begin{itemize}
			\item[(i)] If $\proj_{F} (\xi) = 0$, then $\phi_{\xi}(x) = \Phi_{\xi}$ is constant on $F$, and
			\begin{equation}
			\hat{F} (\xi) = \vol(F) e^{i\Phi_{\xi}}.
			\end{equation}	
			\item[(ii)] If $\proj_{F} (\xi) \neq 0$, then
			\begin{equation}
			\hat{F} (\xi) = \frac{-1}{2 \pi i}  \sum_{G \in \partial F} 
			\frac{\langle \proj_{F} (\xi), N_F(G) \rangle}{\| \proj_{F} (\xi) \|^2} \hat{G} (\xi).
			\end{equation}
		\end{itemize}
\end{thm}

	
\bigskip
We note that our combinatorial Stokes' formula bears some resemblance to Barvinok's version of a combinatorial Stokes' formula, which appeared in \cite{barvinok1}, but has some important differences (see section \ref{remarks} for more details).

In order to account for the entire  chain of terms that arise in the iterative computation of $\hat P$, we  first describe a book-keeping device, called the face-poset of $P$, that will be useful for us in organizing these terms.    We consider chains in the {\bf face poset} $G_P$, the poset of all faces of $P$, ordered by inclusion.  We emphasize that in this paper all chains are always rooted chains, where the root is $P$.  The only appearance of non-rooted chains are in the following definition. 
 If $G$ is a facet of $F$, we attach the following weight to any  (local) chain $(F,G)$, of length $1$,  in the face poset of $P$:

\begin{equation}\label{weight}
 W_{(F,G)}(\xi):=\frac{-1}{2 \pi i} \frac{\langle \proj_{F} (\xi), N_F(G) \rangle }{\| \proj_{F} (\xi) \|^2}. 
 \end{equation}

	Note that these weights are functions of $\xi$ rather than constants. Moreover, they are all homogeneous of degree $-1$.

%

Let $\mathbf{T}$ be any rooted chain in $\GP$, given by 
\[
T:= (P \to F_1 \to F_2, \dots, \to F_{k-1} \to F_k),
\]
so that by definition $\dim(F_j) = d-j$.
We define the {\bf admissible set} $S(\mathbf{T})$ of the rooted chain
 $\mathbf{T}$ to be the set of all vectors  $\xi\in \R^d$ that are orthogonal to the tangent space of $F_k$ but not  orthogonal to the tangent space of  $F_{k-1}$.    
Finally, we define the following weights associated to any such rooted chain $\mathbf{T}$:

\begin{figure}
\centering
     \includegraphics[width=.6\textwidth]{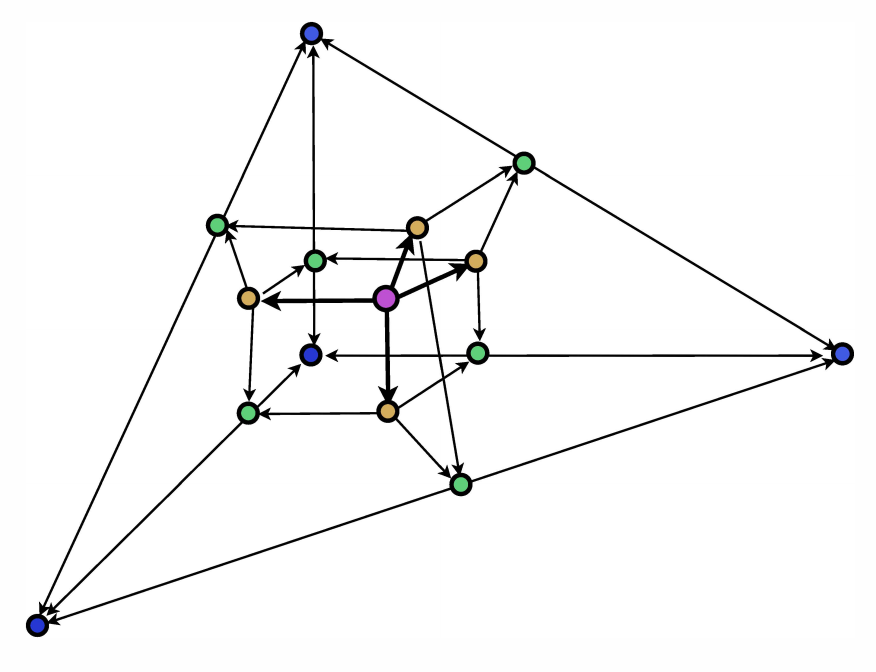}
     \caption{A symbolic depiction of the face poset $G_P$, here drawn as a suggestive directed graph.  We can see all  the rooted chains, beginning from a symbolic vertex in the center, marked with the color purple.  The rooted chains that terminate with the yellow vertices have length $1$, those that terminate with the green vertices have length $2$, and those that terminate with the blue vertices have length $3$. }
\end{figure}

\begin{enumerate}[(a)]
		\item   The rational weight $\mathcal{R}_{\mathbf{T}}(\xi) = \mathcal{R}_{(P \to ... \to F_{k-1} \to F_k)}(\xi)$ is defined to be the product of weights associated to all the rooted chains  $\mathbf{T}$
 of length $1$, times the Hausdorff volume of $F_k$ (the last node of the chain $\mathbf{T}$).  It is clear from this definition that $\mathcal{R}_{\mathbf{T}}(\xi)$ is a homogenous rational function of $\xi$.
		
\bigskip
\item     The exponential weight 
$\mathcal{E}_{\mathbf{T}}(\xi) = \mathcal{E}_{(P \to ... \to F_{k-1} \to F_k)}(\xi)$ 
is defined to be the evaluation of $e^{-2\pi i\langle\xi,x\rangle}$ at any point $x$ on the  face $F_k$:
\begin{equation} \label{exponential.weight}
\mathcal{E}_{\mathbf{T}}(\xi) := e^{-2\pi i\langle\xi,x_0\rangle},
\end{equation}
for any $x_0 \in F_k$.  We note that the inner product $\langle\xi,x_0 \rangle$ does not depend on the position of $x_0 \in F_k$.

\bigskip		
\item      The total weight of a rooted chain $T$ is defined to be

\begin{equation}
W_{\mathbf{T}}(\xi) = W_{(P \to ... \to F_{k-1}  \to F_k)}(\xi):= \mathcal{R}_{\mathbf{T}}(\xi) \mathcal{E}_{\mathbf{T}}(\xi) \mathbf{1}_{S(\mathbf{T})}(\xi), 
\end{equation}

\noindent
where $\mathbf{1}_{S(\mathbf{T})}(\xi)$ is the indicator function of the admissible set $S(\mathbf{T})$ of $\mathbf{T}$.

\end{enumerate} 
	\bigskip
	
\noindent
By repeated applications of the combinatorial Stokes' formula \ref{thm:DSF},  we can then describe the Fourier transform of $P$ as the sum of weights of all the rooted chains of $\GP$:
\begin{align}\label{ingredient1}
		\hat{P}(\xi) = \sum_{\gT} W_{\gT}(\xi) = \sum_{\gT} \cR_{\gT}(\xi) \cE_{\gT}(\xi) \One_{S(\gT)}(\xi).
\end{align}

\bigskip
The main result of this paper is the following explicit description for the coefficients of Macdonald's solid angle sum.	
	
\begin{thm} [Main Theorem]   \label{thm:main} 
		Let $P$ be a $d$-dimensional real polytope in $\R^d$, and let $t$ be a positive real number.
		Then we have
			\[ A_P(t) =\sum_{i = 0}^d a_i(t)t^i, \]
		where, for $0 \leq i \leq d$,
\begin{equation}\label{complicatedcoeff}
 a_i(t) := \lim_{\eps\to 0^+} \sum_{\xi\in\Z^d \cap S(\gT)} 
 \sum_{l(\gT) = d-i} \cR_{\gT}(\xi) \cE_{\gT}(t\xi) \  e^{-\pi\eps\|\xi\|^2},
 \end{equation}
\end{thm}

\noindent
where $l(\gT)$ is the length of the rooted chain $\gT$ in the face poset of $P$, 
$\cR_{\gT}(\xi)$ is the rational function of $\xi$ defined above,  $\cE_{\gT}(t\xi) $ is the complex exponential defined in \eqref{exponential.weight} above, and 
$\Z^d \cap S(\gT)$ is the set of all integer points that are orthogonal to the last node in the chain $T$, but not to any of its previous nodes.

We call the coefficients $a_i(t)$ the {\bf quasi-coefficients} of the solid angle sum $A_P(t)$. 
As a consequence of the main Theorem \ref{thm:main}, it turns out that there is a closed form for the codimension-$1$ quasi-coefficient, which extends previous special cases of this coefficient. 

\begin{thm} \label{codim1coeff}
Let $P$ be any real polytope.  Then the codimension-1 quasi-coefficient of the solid angle sum $A_P(t)$ has the following closed form:
\begin{equation}
a_{d-1} (t) = 
\sum_{\substack{F \textup{ a facet of } P \\ with \ v_F \neq 0}}  \frac{\vol(F)}{\|v_F\|} 
\bar{B}_1 (\langle v_F, x_F \rangle t),
\end{equation}
where $v_F$ is the primitive integer vector which is an outward-pointing normal vector  to $F$.  Also,  $x_F$ is
here any point lying in the affine span of $F$.
\end{thm}

We note that in particular, the formula in Theorem \ref{codim1coeff} shows that the quasi-coefficient  $a_{d-1}(t)$ is always a periodic function of $t$, with a period of $1$, for all real polytopes.
Although it is not true for a  general real polytope that the rest of the quasi-coefficients are periodic functions of $t$, it is true that they are periodic functions in the case of an integer (or rational) polytope, and for all real dilations  $t$, as Theorem \ref{thm:Struc1} below shows.


	The remainder of this paper is organized as follows. Section \ref{section.example} reviews the notation in the context of a simple and concrete example in $\R^2$.   In Section \ref{sec:Fourier}, we discuss how solid angle weights are related to the  Fourier analysis of the polytope and apply the Poisson Summation Formula to transform the solid-angle sum to a limiting sum in the frequency domain. Then, in section \ref{sec:Stokes}, we prove the combinatorial Stokes formula for $\hat P$, which exposes a stratification of the frequency lattice by linear subvarieties on which the Fourier transform of the polytope is representable as products of rational and exponential functions in the frequency vector. 
	In section \ref{ProofOfTheMainTheorem}, we give a proof of the Main Theorem \ref{thm:main}.  
In section  \ref{sec:codim1},  we prove Theorem \ref{codim1coeff}, which gives a closed form formula for the first non-trivial quasi-coefficient of the solid-angle sum, namely $a_{d-1}(t)$.  In section  \ref{sec:apps1} we find a periodicity result for the coefficients of the solid angle polynomial, as an application of the main Theorem.    Finally, in the last section we discuss some of the historical developments of Ehrhart polynomials and Macdonald's solid angle polynomials, and discuss some open problems.

\begin{figure} \label{posetpic}
     \centering
     \includegraphics[width=0.5\textwidth]{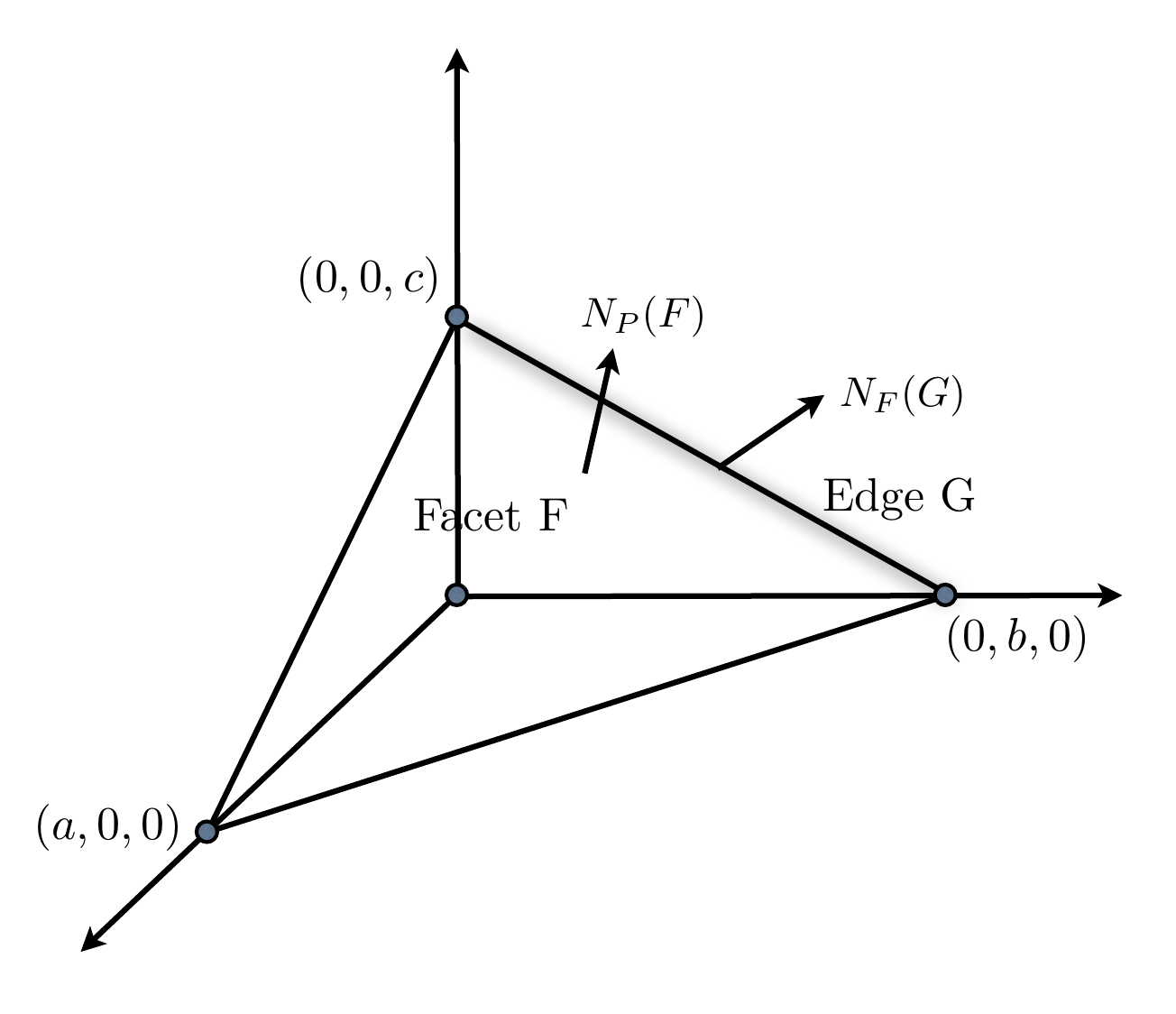}
     \caption{Here $P$ is a tetrahedron, 
        showing one instance of a  rooted     
         chain, defined here by 
         $T := (P~\to~F~\to~G)$, 
         with the normal vectors to each of the relevant faces of $P$}
\end{figure}

\bigskip
\section{An example in three dimensions} \label{section.example}

The following is a concrete example. Consider an integer tetrahedron in $\R^3$, depicted in Figure \ref{posetpic}, whose vertices are 
$(0,0,0), (a,0, 0)$, $(0,b,0)$, and $(0,0,c)$, where $a,b,c$ are integers.  For simplicity, we assume that $\gcd(a,b)= \gcd(a,c) = \gcd(b,c) = 1$.

Following the notation from the figure, the admissible set of vectors $\xi \in \R^3$ for the rooted chain $T:  P \rightarrow F \rightarrow G$ consists of the $2$-dimensional subspace spanned by $N_P(F)$ and $N_F(G)$, `minus'  the $1$-dimensional subspace spanned by $N_P(F)$.  We now compute explicitly the admissible set for the chain $T$.

  Since the equation of the plane that contains $F$  
  is  $\frac{x}{a} +\frac{y}{b}+\frac{z}{c} = 1$, we see that the primitive integer vector in the direction of $N_P(F)$ is  $(\frac{abc}{a}, \frac{abc}{b}, \frac{abc}{c}) = (bc, ac, ab) := v_1$.  Similarly, the primitive integer vector in the direction of $N_F(G)$ can be worked out easily, and is equal to $v_2:= ( a(c^2 + b^2), -bc^2, -b^2c )$.

Thus, for the rooted chain $T$, the admissible set $\Z^3 \cap S(T)$ consists of those integer points (a sublattice of $\Z^3$) contained in the $2$-dimensional subspace spanned by $N_P(F)$ and $N_F(G)$, `minus' the $1$-dimensional integer sublattice containing $N_P(F)$.

To summarize, we have the admissible set $\Z^3 \cap S(T) = \{   m  v_1 + n v_2  \mid    n\not=0, m,n \in \Z   \}$, 
where  $v_1:=  (bc, ac, ab)$ and $v_2:=  ( a(c^2 + b^2), -bc^2, -b^2c )$.

\bigskip

\section{Smoothing estimates for the solid-angle sum}\label{sec:Fourier}

	Let $\One_P$ denote the indicator function of the closed polytope $P$, so that $\One_P(x)=1$ if $x \in P$ and $\One_P(x)=0$ if $x \notin P$. The solid angle with respect to $P$ at any point $x$ in $\R^d$ can be expressed analytically as the convolution between the indicator function $\One_P$ and a Dirac sequence of mollifiers, as we show next.   Although virtually any rapidly decreasing smooth radial function whose total mass is $1$ could be used to construct the Dirac sequence, we have found it particularly convenient to use the $d$-dimensional heat kernels
		$$ \Ge(x)=\eps^{-d/2}e^{-\pi\|x\|^2/\eps} $$
	whose Fourier transform is
		$$ \Geh(\xi)=e^{-\eps\pi\|x\|^2}, $$
and whose normalizing factor $\eps^{-d/2}$ guarantees a total mass of $1$: $\int_{\R^d}\Ge(x)dx=1$. The convolutions of the indicator function $\One_P$ by the heat kernels $\Ge$ will be called the {\bf Gaussian smoothings} of $\One_P$.

\begin{lem} \label{lem:Lem1}
		Let $P$ be a full-dimensional polytope in $\R^d$. Then for each point $x\in\R^d$ 
\begin{equation}
 \lim_{\eps\to 0^+} (\One_P \ast \Ge)(x) = \om_P(x). 
 \end{equation}
\end{lem}
	\begin{proof}
		We have
		\begin{align*}
			(\One_P \ast \Ge) (x) &= \int_{y\in \mathbb{R}^d} \One_P(y) \Ge(x-y) dy \\
								&= \int_{y\in P} \Ge(y-x) dy \\
								&= \int_{u\in P'} \Ge(u) du = \int_{\frac{1}{\sqrt{\eps}}P'} \G_1(v) dv.
		\end{align*}
		
		In the above, we make use of the evenness of $\Ge$ in the second equality. The substitutions $u = y-x$ and $v = u/\sqrt{\eps}$ are also used. Following those substitutions, we change the domain of integration from $P$ to the translation $P'=P- x$ (by the vector $-x$) and to the dilation of $P'$ by the factor $\frac{1}{\sqrt{\eps}}$. When $\eps$ approaches $0$, $\frac{1}{\sqrt{\eps}}P'$ tends to the cone $K$ at the origin subtended by $P'$; $K$ is in fact a translation of the tangent cone of $P$ at $x$. Thus, we arrive at
			\[ \lim_{\eps \to 0^+} (\One_P \ast \Ge) (x) = \int_{K} \G_1(v)dv = \om_K(0) = \om_P(x). \]
	\end{proof}
	
We can now transform the solid-angle sum of $P$ into an infinite sum over the integer lattice, which is
amenable to Fourier-analytic techniques.
	
\begin{thm} \label{thm:SAS1}
		Let $P$ be a full-dimensional polytope in $\R^d$. Then,
		\begin{equation} \label{eq:SAS1}
			A_P(1) = \lim_{\eps\to 0^+} \sum_{x\in\Z^d} (\One_P \ast \Ge)(x). 
		\end{equation}
\end{thm}
	\begin{proof}
		In view of Lemma \ref{lem:Lem1}, we only need to verify the interchange of the limit and the infinite sum. First, we note that
		\begin{align*}
			(\One_P \ast \Ge)(x) &= \int_P \Ge(y-x) dy \leq \sup_{y \in P} \Ge(x-y) \vol(P) = l_{\eps}(x,P) \vol(P).
		\end{align*}

		Here we set $l_{\eps}(x,P):=\sup_{y \in P} \Ge(x-y)$. Let $R$ be a positive number such that 
		$2 \pi R^2 > N$. If $\|x\|^2 > R$ and $0 < \eps \leq 1$, then $\Ge(x)$ is a strictly increasing function of $\eps$. Specifically, $\Ge(x) < \G_1(x)$. Denote $\Omega(R) := \{ x \in \R^d : \sup_{y \in P} \|x-y\|^2 \leq R \}$. We may assume, by increasing $R$ if necessary, that $\Omega(R)$ contains the polytope $P$. Now we have
		\begin{align*}
			\lim_{\eps\to 0^+} \sum_{x\in\Z^d \cap \Omega(R)} (\One_P \ast \Ge)(x) &= \sum_{x\in\Z^d \cap \Omega(R)} \lim_{\eps\to 0^+} (\One_P \ast \Ge)(x) \\
			&=  \sum_{x\in\Z^d \cap \Omega(R)} \omega_P(x) = A_P(1).
		\end{align*}
		The first equality is justified because the inner sum is a finite sum, while the second one follows directly from Lemma \ref{lem:Lem1}. 
		
		We still have to take care of the tail of the infinite sum in Equation (\ref{eq:SAS1}), which is estimated as follows:

		\begin{align*}
			\sum_{x \in \Z^d  \setminus \Omega(R)} (\One_P \ast \Ge)(x) &\leq \vol(P) 
			\sum_{x \in \Z^d \setminus \Omega(R)} l_{\eps}(x,P) \\
			&\leq \vol(P) \sum_{x \in \Z^d \setminus \Omega(R)} l_{1}(x,P).
		\end{align*}

		The usual integral comparison method is now used to dominate this sum by the following integral (taken over a slightly larger region):
			\[ \vol(P) \int_{\R^d \setminus \Omega(R - d^2) } l_{1}(x,P) dx =: g(R), \]
		where $d$ is the diameter of the unit cell. When $R$ tends to $\infty$, $g(R)$ will approach $0$ because $\G_1$ has a finite total mass. This means that the tail of the infinite sum in Equation (\ref{eq:SAS1}) can be bounded by a quantity which depends not on $\eps$, but on $R$, and which tends to $0$ as $R$ goes to $\infty$. Hence, by taking $R$ to be very large, we can obtain an estimate of the right-hand side of Equation (\ref{eq:SAS1}) that is as close to $A_P(1)$ as we want, proving equation (\ref{eq:SAS1}).
	\end{proof}
	
As was mentioned in the introduction, the Poisson Summation formula, applied to the rapidly decreasing function $\One_P \ast \Ge$, yields
\begin{align} \label{convolution} 
\sum_{x\in\Z^d} (\One_P \ast \Ge)(x) &= \sum_{\xi\in\Z^d} \widehat{(\One_P \ast \Ge)}(\xi) \\   
											&= \sum_{\xi\in\Z^d} \hat{\One}_P(\xi) \Geh(\xi) \\
											&= \sum_{\xi\in\Z^d} \hat{\One}_P(\xi) e^{-\pi\eps\|\xi\|^2}.
\end{align}

Thus, taking limits of both sides of \eqref{convolution} as $\epsilon$ tends to zero, and now using the limiting property of Theorem \ref{thm:SAS1}, we obtain
\begin{equation}\label{solidanglesum.1}
A_P(1) =  \lim_{\eps\to 0^+}      \sum_{\xi\in\Z^d} \hat{\One}_P(\xi) e^{-\pi\eps\|\xi\|^2}.
\end{equation}

Going a bit further, we can also easily obtain a similar description for the solid angle sum of any dilate of $P$, by using a standard dilation identity for Fourier transforms:   $\hat{\One}_{tP}(\xi) = t^d \hat{\One}_P(t\xi)$ .

\begin{lem}
Let $P$ be a full-dimensional polytope in $\R^d$ and $t$ any positive real number. 
Then the solid-angle sum of $P$ can be rewritten as follows:
\begin{align} \label{FourierSolidAngleSum}
		A_P(t) = t^d \lim_{\eps\to 0^+} \sum_{\xi\in\Z^d} \hat{\One}_{P}(t\xi) e^{-\pi\eps\|\xi\|^2}, 
\end{align}
\end{lem}
\begin{proof}
Using \eqref{solidanglesum.1}, we have 
		\begin{align}
A_P(t) = A_{tP}(1) &= \lim_{\eps\to 0^+} \sum_{\xi\in\Z^d} \hat{\One}_{tP}(\xi) e^{-\pi\eps\|\xi\|^2} \nonumber\\
		&= t^d \lim_{\eps\to 0^+} \sum_{\xi\in\Z^d} \hat{\One}_{P}(t\xi) e^{-\pi\eps\|\xi\|^2}, \label{eq:DI}
\end{align}

\end{proof}

Using the last Lemma, we may now rewrite the solid angle polynomial in a more suggestive form, as follows. 
\begin{align}\label{polynomiality2}
A_P(t) 
&  =  t^d \lim_{\eps\to 0^+} \sum_{\xi\in\Z^d} 
          \hat{\One}_{P}(t\xi) e^{-\pi\eps\|\xi\|^2} \\
& =   t^d \lim_{\eps\to 0^+} \sum_{\xi\in\Z^d}  
\left( 
\vol(P) e^{i\Phi_{\xi}} \ \One_{\{ 0 \} }(\xi) 
+\ \frac{-1}{2 \pi i} \sum_{G \in \partial P} 
\frac{\langle t\xi, N_P(G) \rangle }{\|  t \xi \|^2} \hat{G} (t \xi) \ 
\One_{\R^d \setminus \{0\}}(\xi)  
\right) \\  \label{lattice-set-up}
& =   t^d   \vol(P)
+ \left( \frac{-1}{2 \pi i}  \right) 
 t^{d-1} \sum_{G \in \partial P} 
 \lim_{\eps\to 0^+} 
 \sum_{\xi\in\Z^d  \setminus \{0\}}
 \frac{\langle \xi, N_P(G) \rangle }{\|  \xi \|^2} \hat{G} (t \xi).
\end{align}

This set-up is now ready for an iterative application of the combinatorial Stokes formula for $\hat F$, applied to each Fourier transform of a facet $G$ (which is a polytope), and so on.

\bigskip
	\section{Proof of the combinatorial Stokes formula for $\hat P$}\label{sec:Stokes}
	
 Here we use the notation $\hat{G}$ with the understanding that the Fourier transform is taken with respect to the Hausdorff measure associated to  $G$, not the Lebesgue measure of the ambient space $\R^d$.

    \begin{proof}
		The gradient $\textup{grad}_{F}\ \phi_{\xi} (x)$ of $\phi_{\xi}(x)$ with respect to the Riemannian structure of the manifold-with-boundary $F$ is proportional to the projection of the gradient of ${\xi}$ onto $F$. More precisely, $\textup{grad}_{F}\ \phi_{\xi} (x) = -2\pi \proj_{F} (\xi)$, which is independent of $x$.

		If $\proj_{F} (\xi) = 0$, then $\phi_{\xi}$ is constant on $F$. Thus, $\hat{F} (\xi)$ is just an integral of a constant function on $F$. If $\Phi_{\xi}$ is the constant value of $\phi_{\xi}$ on $F$, then
			\[ \hat{F} (\xi) = \vol(F) e^{i\Phi_{\xi}}. \]

		For the other case, since $\textup{grad}_{F}\ \phi_{\xi}(x) = \proj_{F} (\xi)$ is constant on $F$, we obtain the identity:
			\[ \textup{div}_F\ \textup{grad}_F\ e^{i\phi_{\xi}(x)} = -\|\textup{grad}_F\ \phi_{\xi}(x)\|^2 e^{i\phi_{\xi}(x)}, \]
		which shows that $e^{i\phi_{\xi}(x)}$ is an eigenfunction of the Laplacian with eigenvalue 
		$$\lambda := -\| \textup{grad}_{F}\ \phi_{\xi}(x)\|^2 \neq 0.$$ 
		Therefore,

		\begin{align*}
			\hat{F}(\xi) &= \int_{F} e^{i\phi_{\xi}(x)} dF 
			= \frac{1}{\lambda} \int_{F} \textup{div}_{F}\ \textup{grad}_{F}\ e^{i\phi_{\xi}(x)} dF \\
						&= \frac{1}{\lambda} \sum_{G \in \partial F} \int_{G} \langle \textup{grad}_{F} e^{i\phi_{\xi}(x)},N_F(G) \rangle dG \\
						&= \frac{i}{\lambda} \sum_{G \in \partial F} \langle \textup{grad}_{F}\  \phi_{\xi}(x),N_F(G) \rangle \int_{G} e^{i\phi_{\xi}(x)} dG \\
						&= \frac{-1}{2 \pi i} \sum_{G \in \partial F} \frac{\langle \proj_{F} (\xi), N_F(G) \rangle }{\| \proj_{F} (\xi) \|^2} \hat{G} (\xi).
		\end{align*}
We have employed Stokes' theorem on $F$ in the third equality, writing its Fourier transform as a finite weighted combination of Fourier transforms of its facets.
	\end{proof}

\medskip
We notice that when $\dim(F) = 1$, which means $F$ is a line segment, the above theorem is still valid with the convention that the Fourier transforms of its facets, which are vertices in this case, are just the valuations of $e^{i\phi_{\xi}(x)}$ at those vertices.

	This theorem will be our main tool for understanding the behavior of the solid-angle sums $A_P(t)$ for all nonzero real numbers $t$.  We denote the subspace of $\R^d$ orthogonal to a polytope $F$ by $F^{\perp}$. By convention, if $\dim(F)=0$ then $F^{\perp} = \R^d$.

Thus, Theorem \ref{thm:DSF}  may be rewritten as
\begin{align} \label{DSFrewritten}
&	\hat{F}(\xi)\ =\ \vol(F) e^{i\Phi_{\xi}} \ \One_{F^{\perp}}(\xi) 
			+\ \frac{-1}{2 \pi i} \sum_{G \in \partial F} \frac{\langle \proj_{F} (\xi), N_F(G) \rangle }{\| \proj_{F} (\xi) \|^2} \hat{G} (\xi) \ \One_{\R^d \setminus F^{\perp}}(\xi).
\end{align}

We have written the combinatorial Stokes' formula in general here, meaning that we can now apply it to any face $F$ of $P$.  But to begin the iteration, we would like to to apply our combinatorial Stokes' formula  \ref{thm:DSF} first with $F:= P$, a full-dimensional real polytope in $\R^d$.   In this case full-dimensional case, we have to take care of those $\xi \in \R^d$ which are orthogonal to all of $P$, but this implies that there is only the $0$-dimensional subspace here, namely $\xi = 0$.  Also, noting that in this first full-dimensional step we have simply $\proj_{P} (\xi) = \xi$, \eqref{DSFrewritten} becomes

\begin{align} \label{FirstIterationOfStokes}
&   \hat{P}(\xi)\ =    \vol(P) e^{i\Phi_{\xi}} \ \One_{\{ 0 \} }(\xi) 
+\ \frac{-1}{2 \pi i} \sum_{G \in \partial P} 
\frac{\langle \xi, N_P(G) \rangle }{\|  \xi \|^2} \hat{G} (\xi) \ 
\One_{\R^d \setminus \{0\}}(\xi).
\end{align}

\bigskip
\section{Proof of Theorem \ref{thm:main} }    \label{ProofOfTheMainTheorem}

\begin{proof}
            Beginning with equation \eqref{FourierSolidAngleSum}, we have
\begin{align*}
			A_P(t) &=  t^d \lim_{\eps\to 0^+} \sum_{\xi\in\Z^d} \hat{\One}_{P}(t\xi) e^{-\pi\eps\|\xi\|^2} \\
					&= t^d \lim_{\eps\to 0^+} \sum_{\xi\in\Z^d} \sum_{\gT} \cR_{\gT}(t\xi) \cE_{\gT}(t\xi) \One_{S(\gT)}(t\xi) e^{-\pi\eps\|\xi\|^2} \\
					&= t^d \lim_{\eps\to 0^+} \sum_{\xi\in\Z^d} \sum_{\gT} \cR_{\gT}(\xi) t^{-l(\gT)} \cE_{\gT}(t\xi) \One_{S(\gT)}(\xi) e^{-\pi\eps\|\xi\|^2} \\
					&= t^d \lim_{\eps\to 0^+} \sum_{\xi\in\Z^d \cap S(T)} \sum_{\gT} \cR_{\gT}(\xi) t^{-l(\gT)} \cE_{\gT}(t\xi)    e^{-\pi\eps\|\xi\|^2}.
\end{align*}

		Here, the inner sums in the last two lines are taken over all rooted chains of           $\GP$.   We used \eqref{ingredient1} in the second equality, noting that this is the step that uses the face poset of $P$. 
		The last equality follows from the homogeneity of $\cR_{\gT}(\xi)$ and the fact that a point $\xi$ lies in the admissible set $S(\gT)$ if and only if $t\xi$ also lies in $S(\gT)$.   We may now group together the rooted chains $\gT$ by the length $l(\gT)$ of $\gT$,  and  we obtain the desired identity.
	\end{proof}


\bigskip
\section{The quasi-coefficient $a_{d-1}(t)$} \label{sec:codim1}

In this section we work with any  {\it real} polytope $P$, and we prove Theorem \ref{codim1coeff}.   Although the 
quasi-coefficients of $A_P(t)$ appear to be difficult to calculate in general, it turns out, somewhat surprisingly,  that there exists a closed-form for the codimension-1 quasi-coefficient $a_{d-1}(t)$,  for all {\it real} dilations of $P$. 
We first recall the description of $a_{d-1}(t)$, given by the main Theorem \ref{thm:main} :
\begin{align*}
a_{d-1}(t) &:= \lim_{\eps\to 0^+} \sum_{\xi\in\Z^d} \sum_{l(\gT) = 1} \cR_{\gT}(\xi) \cE_{\gT}(t\xi) \One_{S(\gT)}(\xi) e^{-\pi\eps\|\xi\|^2} \\
&= \sum_{F \textup{ facet of } P} \lim_{\eps\to 0^+} \sum_{\xi \in (\Z^d \cap F^{\perp}) \setminus 0}  \cR_{P \to F}(\xi) \cE_{P \to F}(t\xi) e^{-\pi\eps\|\xi\|^2} 
\end{align*}
In the last equality, we use the fact that the admissible set for the rooted chain $(P \to F)$ is $F^{\perp} \setminus 0$.  Note that $F^{\perp}$ has dimension $1$, since it is spanned by the normal vector to $F$, $N_P(F)$.  Thus one of the following two cases must occur:

\medskip \noindent 
Case 1.  \ $F^{\perp}$ contains a $1$-dimensional sublattice of $\Z^d$, or
 
 \medskip \noindent
Case 2.   \ $F^{\perp}$ only has the origin in its intersection with $\Z^d$.
 
\medskip  \noindent
In case (1), we define $v_F$ to be the unique lattice point of shortest Euclidean norm with the same direction as the outward-pointing unit normal vector $N_P(F)$ (also called a primitive lattice point).  In case (2), we set $v_F = 0$.  We note that case (2) may happen if  for example the components of the normal vector are irrationals that are pairwise incommensurable with one other.  

Suppose we are in case (1). Then any frequency vector $\xi \in (\Z^d \cap F^{\perp}) \setminus 0$ must have the form $n v_F$ for some nonzero positive integer $n$. Let $x_F$ be an arbitrary point in the affine span of  $F$, and set
\begin{equation}
h_F(t) := \langle v_F, x_F \rangle t.
\end{equation}

\noindent
We have
\begin{align*}
\mathcal{R}_{(P \to F)}(\xi) &= \frac{\vol(F)}{-2 \pi i}\frac{\langle \xi, \frac{v_F}{\|v_F\|} \rangle}{\|\xi\|^2} = \frac{\vol(F)}{-2 \pi i}\frac{n \|v_F\|}{\|n v_F\|^2} =
\frac{1}{-2 \pi i n}   \frac{\vol(F)}{\|v_F\|}, \\
\mathcal{E}_{(P \to F)}(t\xi) &= e^{-2 \pi i n (\langle v_F, x_F \rangle t)} = e^{-2 \pi i h_F(t) n}.\\
\end{align*}
In case (2), the summation domain $(\Z^d \cap F^{\perp}) \setminus 0$ is empty. Hence, we can omit this case in the computation of $a_{d-1}(t)$ and deduce
\begin{align} \label{codim1coeff.setup}
a_{d-1}(t) &= \sum_{\substack{F \textup{ facet of } P \\ v_F \neq 0}}  \frac{\vol(F)}{\|v_F\|} \left( \lim_{\eps\to 0^+} \sum_{n \in \Z_{\neq 0}}  \frac{e^{-2 \pi i h_F(t) n}}{-2 \pi i n} e^{-\pi \eps \|v_F\|^2 n^2} \right) \nonumber \\
&= \sum_{\substack{F \textup{ facet of } P \\ v_F \neq 0}}  \frac{\vol(F)}{\|v_F\|} \left( \lim_{\eps\to 0^+} \sum_{n \in \Z_{\neq 0}}  \frac{e^{-2 \pi i h_F(t) n}}{-2 \pi i n} e^{-\pi \eps n^2} \right). 
\end{align}

We first focus on a fixed facet $F$ with $v_F \neq 0$, which simply means that the orthogonal complement of $F$ contains a non-degenerate sublattice of $\Z^d$. The key to computing the inner limiting sum corresponding to $F$ is to realize that its summands can be expressed as the Fourier transform of a compactly supported function and then to re-apply the Poisson summation formula  (in reverse, so to speak).  It turns out that the required function here is a translation of the first Bernoulli polynomial $B_1(x)$. 
Recall that $B_1(x)$ is defined as 
\begin{equation}\label{eq:Bernoulli1}
B_1(x) := \left\{
\begin{array}{ll}
x - \frac{1}{2} & \textup{when } x \in (0,1), \\
0 & \textup{otherwise}.
\end{array}
\right. 
\end{equation} 
Its Fourier transform, evaluated at any integer frequency $n \in \Z$, is
\begin{align*}
\widehat{B_1} (n) &= \int_0^1 \left( x - \frac{1}{2} \right) e^{2 \pi i x n} dx \\
&= \left\{
\begin{array}{ll}
\frac{1}{2 \pi i n} & \textup{when } n \in \Z_{\neq 0}, \\
0 & \textup{when } n = 0.
\end{array}
\right.
\end{align*}

Let $T_{x_0} (x) := x - x_0$ denote the translation by $x_0$. Then, 
\begin{align*}
\widehat{B_1 \circ T_{h_F(t)}} (n) &= \widehat{B_1}(n) e^{-2\pi i h_F(t) n} \\
&= \left\{
\begin{array}{ll}
\frac{e^{-2\pi i h_F(t) n}}{2 \pi i n} & \textup{when } n \in \Z_{\neq 0}, \\
0 & \textup{when } n = 0.
\end{array}
\right.
\end{align*}

\noindent
The remaining factor of the innermost summand in Equation \eqref{eq:codim1} is simply the Fourier transform of the $1$-dimensional Gaussian $\Ge(x)$. In order to finish the computation of the codimension-1 quasi-coefficient $a_{d-1}(t)$, we need to slightly generalize both Lemma \ref{lem:Lem1} and Theorem \ref{thm:SAS1} by replacing the indicator function $\One_P$ with a continuous function $f$ which is
supported on the polytope $P$.   Because $P$ is bounded, $f$ is both uniformly continuous and compactly supported.  This prevents any convergence issues and allows us to interchange the limit and the summation, as in the proofs of Lemma \ref{lem:Lem1} and Theorem \ref{thm:SAS1}. Hence, we obtain the following results, whose proofs are identical with our previous proofs.

\begin{lem} \label{lem:Lem2}
	Let $f$ be a continuous function on $P$ and zero outside $P$. Then, for all $x \in \R^d$,
	\[ \lim_{\eps \to 0^+} (f \ast \Ge) (x) = f(x) \omega_P(x). \]
\end{lem}

\begin{thm} \label{thm:SAS2}
	Let $P$ be a full-dimensional polytope in $\R^d$. Then,
	\[ \sum_{x \in \Z^d} f(x) \omega_P(x) = \lim_{\eps \to 0^+} \sum_{x \in \Z^d} (f \ast \Ge) (x). \]
	
\end{thm}

\noindent
Note that the left-hand side of the above identity is in fact a finite sum because $P$ is compact. Now we can continue our computation as follows:

\begin{align*}
\lim_{\eps\to 0^+} \sum_{n \in \Z_{\neq 0}}  \frac{e^{-2 \pi i h_F(t) n}}{-2 \pi i n} e^{-\pi \eps n^2} &= \lim_{\eps\to 0^+} \sum_{n \in \Z} 
       \mathcal{F} \{   B_1 \circ T_{h_F(t)}   \} (n) \Geh(n) \\
&= \lim_{\eps\to 0^+} \sum_{n \in \Z}  \mathcal{F} \{ (B_1 \circ T_{h_F(t)}) \ast \Ge \}(n) \\
&= \lim_{\eps\to 0^+} \sum_{n \in \Z}  \left( (B_1 \circ T_{h_F(t)}) \ast \Ge \right) (n)  \\
&= \sum_{n \in \Z}  (B_1 \circ T_{h_F(t)})(n) \omega_{[h_F(t),1 + h_F(t)]} (n) \\
&= \bar{B}_1 (h_F(t)),
\end{align*}	

\noindent
where we have used Poisson summation in the third equality above. 
Here, $\bar{B}_1 (x)$ is the periodized version of   $B_1$, and is defined formally
 as $\bar{B}_1 (x) = B_1(\{x\})$ with $\{x\} = x - \lfloor x \rfloor$ being the fractional part of $x$. In the above calculation, the fourth  equality  follows from Theorem \ref{thm:SAS2}. The last equality can be easily derived by considering separately the case when $h_F(t) :=  \langle v_F, x_F \rangle t$ is an integer or not. 
 
 \noindent
 Therefore, from \eqref{codim1coeff.setup}, 
 we finally obtain the required explicit formula of Theorem~\ref{codim1coeff}
 for the codimension-1 quasi-coefficient of the solid-angle sum $A_P(t)$:
 \begin{equation*}
a_{d-1} (t) = 
\sum_{\substack{F \textup{ a facet of } P \\ with \ v_F \neq 0}}  \frac{\vol(F)}{\|v_F\|} 
\bar{B}_1 (\langle v_F, x_F \rangle t),
\end{equation*}
 where $x_F$ is any point lying in the affine span of $F$, and $v_F$ is the unique primitive integer vector which is an outward-pointing normal vector  to the facet $F$.


\bigskip
\section{Applications of the Fourier-analytic description}   \label{sec:apps1}
	
We again recall that it is classically known, from the work of Macdonald \cite{macdonald1, macdonald2} that, when restricted to integer dilates $t$, the solid angle sum $A_P(t)$ is a polynomial if $P$ is an integer polytope and a quasi-polynomial if $P$ is a rational polytope. The following theorem generalizes both cases and reveals the periodicity of all quasi-coefficients $a_i(t)$, for $0 \leq i \leq d$, in the case of {\it real} dilates of an integer polytope $P$. In fact, they share a common period.

\begin{thm}[Periodicity] \label{thm:Struc1}
		Let $P$ be a full-dimensional integer polytope in $\R^d$,  and let $t$ be a nonzero real number. Then, for all $0 \leq i \leq d$, the quasi-coefficient $a_i(t)$ is a periodic function, with period $1$.
\end{thm}
	\begin{proof}
		In fact, it follows from Theorem \ref{thm:main} that
			\[ a_i(t) := \lim_{\eps\to 0^+} \sum_{\xi\in\Z^d \cap S(\gT)} 
			\sum_{l(\gT) = d-i} \cR_{\gT}(\xi) \cE_{\gT}(t\xi) \  e^{-\pi\eps\|\xi\|^2} \]
		for $0 \leq i \leq d$. Notice that, in the above equation, the variable $t$ only appears in the exponential functions $\cE_{\gT}(t\xi)$. These functions can be seen to be periodic in $t$ with period equal to $1$ as follows.
		
\begin{equation}\label{TrivialExponentialTerms}
\cE_{\gT}((t+1)\xi) :=  e^{i\phi_{(t+1)\xi}(x)} = e^{-2\pi i \langle (t+1)\xi,x \rangle} = e^{-2\pi i \langle t\xi,x \rangle} e^{-2\pi i \langle \xi,x \rangle}. 
 \end{equation}

		Let $F$ be the terminal face of the rooted chain $\bf T$. Recall that $x$ can be chosen arbitrarily on the affine subspace spanned by $F$. Therefore, we can set $x$ to be one of the vertices of $F$, which are lattice points. Hence, both $\xi$ and $x$ have integer coordinates, which forces $e^{-2\pi i \langle \xi,x \rangle} = 1$. Hence,
		\[ \cE_{\gT}((t+1)\xi) = e^{-2\pi i \langle t\xi,x \rangle} = \cE_{\gT}(t \xi). \]
This completes our proof of the theorem.
\end{proof}

\medskip
To see that the above theorem covers the case of any rational polytope $P$, let $Q = nP$ be a dilation of $P$ such that $Q$ is now an integer polytope. Below we add the superscripts $P$ and $Q$ to distinguish between the quasi-coefficients of $A_P(t)$ and those of $A_Q(t)$. Therefore, for $0 \leq i \leq d$ and $t \in \Z_{\neq 0}$,
\[ a_i^P(t+n) = a_i^Q\left(\frac{t}{n} + 1\right) = a_i^Q\left(\frac{t}{n}\right) = a_i^P(t). \]
This shows that the solid angle sum $A_P(t)$ for integer dilates of the rational polytope $P$ is a {\bf quasi-polynomial} in the {\bf classical } sense:
 \begin{equation} \label{quasipolynomial}
A_P(t) = (\vol P)  t^d + a_{d-1}(t)  t^{d-1} + a_{d-2}(t)  t^{d-2} +  \cdots +  a_0(t), 
\end{equation} 
where each quasi-coefficient $a_j(t)$ is now a periodic function of $t \in \Z_{>0}$.

We note that the quasi-coefficients in Theorem \ref{thm:Struc1} above might sometimes have a smaller period than a period equal to $1$, and they might even have a period of $0$, which means we would have a polynomial $A_P(t)$, even though $P$ would have some rational coordinates for its vertices.   In general such a phenomenon is called {\bf period collapse} in the case of rational polytopes (\cite{BeckSamWoods},\cite{haasemcallister}, \cite{mcallisterwoods}).   Recently, this phenomenon of period collapse has been studied in the context of triangles with quadratic irrational slopes \cite{CLS}.

\bigskip
\section{Retrieving classical results from the main theorem}\label{classical}
	
Using our main result, namely Theorem \ref{thm:main}, we can easily recover some of the classical theory of solid angles.  Here we offer another proof of Macdonald's classical result, namely equation \eqref{solidanglesum2}, for the solid angle polynomial in the case of an integer polytope.  

\begin{cor}
Let $P$ be a $d$-dimensional integer polytope in $\R^d$.  Then Macdonald's solid angle sum $A_P(t)$ is a polynomial when $t$ is restricted to be a positive integer.  Moreover, if $\dim P$ is even, then $A_P(t)$ is an even polynomial in the positive integer parameter $t$, and if $\dim P$ is odd, then $A_P(t)$ is an odd polynomial in the positive integer parameter $t$.
\end{cor}
\begin{proof}
As was already noted in equation \eqref{TrivialExponentialTerms}, all of the exponential weights for all rooted chains are trivial, namely equal to $1$, because all faces are integer polytopes (of various dimensions).  Using the first iteration of our combinatorial Stokes' formula, namely \eqref{FirstIterationOfStokes}, and plugging it into Poisson summation, we have, from \eqref{lattice-set-up}:
\begin{align}\label{polynomiality}
A_P(t) & =   t^d   \vol(P)
+ \left( \frac{-1}{2 \pi i}  \right) 
 t^{d-1} \sum_{G \in \partial P} 
 \lim_{\eps\to 0^+} 
 \sum_{\xi\in\Z^d  \setminus \{0\}}
 \frac{\langle \xi, N_P(G) \rangle }{\|  \xi \|^2} \hat{G} (t \xi).
\end{align}
Note the appearance of the purely imaginary number $\frac{-1}{2 \pi i}$.    Now, when we iterate the combinatorial Stokes' formula again, this time applying it to
 $\hat G(t \xi)$, we obtain another linear combination of volumes of facets of $G$ (hence codimension $2$ faces of $P$), and another linear combination of the pure imaginary term $\frac{-1}{2 \pi i}$, multiplied by the real rational weights times the transforms of the codimension-$2$ faces of $P$.  When we iterate the combinatorial Stokes' formula all the way down to the vertices, we see that every coefficient of $t^{d-k}$ in this expansion will be purely imaginary if $k$ is odd, and every coefficient of $t^{d-k}$ will be real if $k$ is even.  Since the left-hand side of   \eqref{polynomiality} is real, by definition of $A_P(t)$, we arrive at the result.
\end{proof}

\begin{cor} \label{constant.term.vanishes}
Let $P$ be a $d$-dimensional integer polytope in $\R^d$.  Then the constant coefficient $a_0(t)$ in the solid angle sum $A_P(t)$, for positive integer values of $t$, is identically equal to zero. 
\end{cor}
\begin{proof}
Again as noted in equation \eqref{TrivialExponentialTerms},  when $t$ is an integer all the exponential terms $\cE_{\gT}(t\xi)$ degenerate to $1$.   Moreover, we can pair up the two paths $\gT_1,\gT_2$ of length $n$ if they are almost the same except for their last nodes $V_1 \neq V_2$, which are 0-dimensional faces of $P$. This happens only when $V_1,V_2$ are the two ends of a common $1$-dimensional edge $E$ of $P$. The only difference in the weights of $\gT_1,\gT_2$ are the weights of the edges $(E,V_1)$ and $(E,V_2)$ of the graph $\GP$. The fact that $N_E(V_1) = -N_E(V_2)$ implies that, from definition \eqref{weight},
		\begin{align*}
			W_{(E,V_1)} (\xi) &= \frac{-1}{2 \pi i} \frac{\langle \proj_{E} (\xi), N_E(V_1) \rangle }{\| \proj_{F} (\xi) \|^2} \\
							&= - \frac{-1}{2 \pi i} \frac{\langle \proj_{E} (\xi), N_E(V_2) \rangle }{\| \proj_{F} (\xi) \|^2} \\
							&= - W_{(E,V_2)} (\xi).
		\end{align*}

\noindent 
Thus, $\cR_{\gT_1}(\xi) = -\cR_{\gT_2}(\xi)$. Moreover, $\gT_1$ and $\gT_2$ have the same admissible set $\R^d \setminus E^{\perp}$. Hence $W_{\gT_1}(\xi) = -W_{\gT_2}(\xi)$ for all $\xi$ in the admissible set of both rooted chains $T_1$ and $T_2$, and so the constant coefficient $a_0(t)$ is zero for positive integers $t$.
\end{proof}

Next, we give a third quick corollary of the main result, this time retrieving a reciprocity law for the solid angle polynomial of a real polytope, which appeared in \cite{desariorobins}.   Looking back at equation  \eqref{FourierSolidAngleSum} allows us to extend the domain of the  solid-angle sum to all nonzero real numbers $t \in \R$.  
 In other words,  we can extend the function $A_P(t)$, already defined for all positive reals $t$ in \eqref{FourierSolidAngleSum}, to include all nonzero 
 $t \in \R$:
\begin{equation}\label{extension}
A_P(t) :=  t^d \lim_{\eps\to 0^+} \sum_{\xi\in\Z^d} \hat{\One}_{P}(t\xi) 
e^{-\pi\eps\|\xi\|^2},
\end{equation}

\noindent 
We obtain the following reciprocity law, which appeared in \cite{desariorobins} for real polytopes.  The methods of \cite{desariorobins} also used the Poisson summation, but did not have a detailed description of the quasi-coefficients of $A_P(t)$. 

\begin{cor}[Reciprocity Law for Real Polytopes] \label{thm:Rec}
		Let $P$ be a full-dimensional real polytope in $\R^d$. Then the extended solid angle sum
	\eqref{extension} satisfies the functional identity:
\begin{equation}
A_P(-t) = (-1)^{\dim(P)} A_P(t),
\end{equation}
valid for all nonzero real numbers $t$.
\end{cor}

	\begin{proof}
		Equation (\ref{eq:DI}) gives us the following functional identity, valid for each nonzero real number $t$:
		\begin{align*}
		A_P(-t) &= (-t)^d \lim_{\eps\to 0^+} \sum_{\xi\in\Z^d} \hat{\One}_{P}(-t\xi) e^{-\pi\eps\|\xi\|^2} \\
	&= (-1)^d  t^d \lim_{\eps\to 0^+} \sum_{m \in\Z^d} \hat{\One}_{P}(tm) e^{-\pi\eps\| m \|^2} \\
	&= (-1)^d  A_P(t),
		\end{align*}
		where $m = - \xi$ was used in the penultimate equality.
\end{proof}

To state the next very classical result, we need to define 
the solid angle of a face $F$ of $P$: it is defined by choosing any point $x$ in the relative interior of $F$ and setting $\omega_P(F) := \omega_P(x)$.   For example, a $1$-dimensional edge $E$ of a $3$-dimensional polytope $P$ has a solid angle associated to it which happens to be equal to the dihedral angle that $E$ makes with its two adjoining facets of $P$.

\begin{cor}[Brianchon-Gram relations] \label{GramRelations}
For a rational convex polytope $P$, we have 
\begin{equation}
\sum_{F \subset P} (-1)^{\dim F} \omega_P(F) = 0.
\end{equation}
\end{cor}
\begin{proof}
First we note that the solid angle of each face remains invariant under dilations, so that we may assume that $P$ is now an integer polytope, after an appropriate integer dilation of the given rational polytope.   In this case we now know that $A_P(t)$ is a polynomial in the integer parameter $t$, and we also have the Ehrhart polynomial of the relative interior of any face F of $P$, which we call $L_{F^o}(t)$.  We make use of the elementary but useful relation (see \cite{BeckRobins}, Lemma 13.2) between the Solid angle polynomial of $P$ and the Ehrhart polynomials of the relatively open faces of $P$:
\begin{equation}\label{SolidAngle.Ehrhart}
A_P(t) = \sum_{F \subset P} \omega_P(F) L_{F^o}(t).
\end{equation}
Considering the constant terms of all polynomials in $t \in \Z$ on both sides of \eqref{SolidAngle.Ehrhart},  using the fact that the constant term of each $L_{P^o}(t)$ is equal to $(-1)^{\dim F}$, and using our Corollary \ref{constant.term.vanishes} (the constant term of the solid angle polynomial always vanishes for integers $t$), we have proved the Brianchon-Gram identity. 
\end{proof}


\bigskip
\section{Further remarks}\label{remarks}

In the 1970's, Peter McMullen \cite{McMullen1, McMullen2}  found proofs, and generalizations, of the structural properties of Macdonald's solid angle polynomial, using the theory of valuations.  The proofs tend to be easier, and are beautifully general, once the valuations are developed.  A more precise formulation, and algorithms for, the quasi-coefficients of $A_P(t)$ still appears to be out of reach, though, using valuations alone.  Lawrence \cite{lawrence} has helped develop the theory of valuations as well, and has contributed to the computational complexity analysis of the volumes of polytopes.    

Sometimes the relationship \eqref{SolidAngle.Ehrhart} allows us to transfer information from Ehrhart polynomials to solid angle polynomials, but we note that in practice much is left to be desired from such a simple dictionary, due to the immense complexity involved in computing all of the Ehrhart polynomials $L_{F^o}(t)$ for \emph{all} faces of $P$.  
General relations between solid angles of the faces of a polytope (as in the Brianchon-Gram relations) were studied by Micha Perles and Geoffrey C. Shepard \cite{perles}, \cite{shephard}.  Perhaps the earliest account of relations between solid angles of faces of a polytope were given by the great 19'th century geometer Schl\"afli \cite{schlafli}.  In the early 1900's  Sommerville \cite{sommerville} took up this study.  More recently, further generalizaions were obtained by Kristin Camenga  \cite{camengasolidangles}.

Although formula \eqref{complicatedcoeff}, for the quasi-coefficients, looks somewhat complicated, in low dimensions it is indeed possible to push the computations of these lattice sums through, as has recently been accomplished in \cite{Nhat.Sinai}, for $\R^2$.  

We note that while Theorem \ref{thm:main} is valid for any real polytope, it might be tempting to conclude that the quasi-coefficients $a_i(t)$ are periodic for any real polytope, but this conclusion is false in general.  Indeed it is  a very difficult question to find the order of growth of $a_i(t)$ for a general real polytope and is a fascinating open question. 

Regarding the computation of one solid angle in $\R^d$, Ribando \cite{ribando} gave an interesting method, via hypergeometric series, to measure one solid angle.  The optimal computational complexity of computing one solid angle of even a simplicial cone in $\R^d$ is an open problem. 


Further complexity considerations for the computation of Ehrhart polynomials and fixed Ehrhart coefficients were carried out initially by Barvinok in \cite{barvinok2}, with an exposition in \cite{barvinok3}.  Barvinok's algorithm shows that, for fixed dimension $d$, it is possible to compute the number of integer points in a rational polytope in polynomial-time, as a function of the bit-capacity of the coordinates of the polytope.  

In \cite{BeckRobinsSam}, the authors study solid angle polynomials from the perspective of analogues of Stanley's non-negativity results of the Ehrhart series and obtain similar results.    In \cite{desariorobins}, the authors study solid angle polynomials and extend their domain to all real numbers, but do not obtain the precise formulations that we obtain here for the quasi-coefficients of the solid angle polynomial. 

The work of Pommersheim \cite{pommersheim} from $1993$ initiated a study of Ehrhart polynomials of integer polytopes through the interesting use of Todd classes of Toric varieties.    We continue to see many works on these relationships (see \cite{morelli} as well), and the interested reader may also consult Fulton's book \cite{fulton}, Danilov's survey paper \cite{danilov}, or a survey paper by Barvinok and Pommersheim \cite{barvinokpommersheim}  for a dictionary between some of the enumerative geometry of polytopes (including lattice point enumeration) and some algebraic properties of Toric varieties.

Yau and Zhang  \cite{YauZhang} found an interesting upper bound on the number of lattice points in any real right-angled simplex in $\R^d$, and used this bound to prove the difficult Durfree conjecture in algebraic geometry, for isolated weighted homogeneous singularities.  Ehrhart polynomials of simplices also find applications in the famous ``linear diophantine problem of Frobenius'', also known as the money-exchange problem \cite{BeckDiazRobins}, \cite{BeckRobinsPolygons}, \cite{kannan}, \cite{shallit} in number theory.   The formulas in \cite{BeckDiazRobins}, \cite{BeckRobinsPolygons} gave rise to Ehrhart quasi-polynomials of rational simplices, whose `atomic pieces' are composed of generalized Dedekind sums, a vast and fascinating area of number theory, topology, and pseudo-random number generators  \cite{Knuth}.

Skriganov \cite{Skriganov1, Skriganov2, Skriganov3} has studied the difference between the volume of a real convex polytope $P$ and its lattice-point enumerator 
$rP \cap \mathcal L$, for various lattices $\mathcal L$.   His methods include the Poisson summation formula, but take a very different route than our methods.  In particular, Skriganov uses an ergodic approach on the space of lattices 
$SL_d(\R) / SL_d(\Z)$.   He proves that for large classes of real polytopes $P$, the error term 
$|rP \cap \mathcal L| -  (\vol P) r^d$ is extremely small, on the order of  
$O(\log r)^{d-1 + \epsilon }$. 

Regarding the combinatorial Stokes' formula of Barvinok, from \cite{barvinok1}, the difference between Barvinok's formula and our combinatorial Stokes' formula is that he uses a complex parameter which renders the denominators non-vanishing.  Although this appears to be a desirable feature, it is not at all clear how one might derive the grading that we have found (or if it is possible), namely the quasi-coefficient formulas, using Barvinok's combinatorial Stokes' formula. This might be an interesting avenue for further research.  
	
In the field of discrete optimization, Ehrhart quasi-polynomials play a prominent role as well now - see \cite{deloerahemmeckekoeppe}.
Such questions are natural, given the need to optimize linear functionals over the integer lattice, as in integer linear programming \cite{schrijver}.   In the analysis of lattices and their invariants, solid angles can be used to measure how `short' a basis of a lattice can optimally be \cite{fukshanskyrobinssolidangle}.

Recently, the work of Stapledon \cite{stapledonadditive, stapledonweightedehrart} and Stapledon and Katz \cite{Katz.Stapledon} used ``local h-vectors'', which is based on Richard Stanley's subdivision theory \cite{Stanley1}, and has found very nice applications to the solution of various open problems concerning unimodality questions in enumerative geometry, and mixed hodge numbers.  The usual h-vectors are formed by the numerator polynomial of the Ehrhart series, which is a change-of-basis of the Ehrhart polynomial, so these recent advances are extensions of Ehrhart polynomials.

The recent work of Eva Linke \cite{linke} on Ehrhart polynomials shows a remarkable property of the Ehrhart coefficients:  they satisfy a linear ordinary differential equation when passing from the $k$'th quasi-coefficient to the $(k+1)$'st quasi-coefficient.   One might wonder if it is possible to use such an ODE to compute Ehrhart coefficients more efficiently, but such a direction still appears to be out of reach, perhaps due to the difficulty in finding the requisite initial conditions for these ODE's. 

In the field of tiling and multi-tiling, solid angles play a role \cite{gravinrobinsshiryaev} in giving an equivalent condition for a rational polytope $P$ to be able to multi-tile $\R^d$ by translations with a lattice (and more general sets).  This condition is essentially equivalent to saying that the sum of the solid angles of $P$, taken at all integer points, equals the volume of $P$, and this hold for all translations of $P$.   There are many open problems in this area (see \cite{gravinrobinsshiryaev} for some open problems).

Going back further in time, in 1922 Hardy and Littlewood wrote about    \cite{HardyLittlewood} ``Some problems of Diophantine approximation: The lattice-points of a right-angled triangle'', and they used weights of $1/2$ for the lattice points on the boundary of their triangles, which means they considered the solid angle sum, in an apparently ad-hoc way, predating the results of Macdonald from the 1960's. 

Payne \cite{payne}  has studied combinatorial relationships between the Ehrhart series of lattice polytopes and various conjectures of Hibi, Stanley, and others regarding unimodality of the $h$-polynomials, for Ehrhart series.   Such questions are again related to whether or not a polytope is reflexive, and hence are related to questions of Batyrev.  

The work of Batyrev \cite{Batyrev} on mirror symmetry for Calabi-Yau hypersurfaces in toric varieties shows how useful Ehrhart theory can be for the study of relations between some integer polytopes, called reflexive polytopes, and their duals.  A reflexive polytope may be defined (there are many equivalent but apriori distinct definitions) as an integer polytope such that between any two of its consecutive integer dilations, there are no other integer points (see \cite{BeckRobins}). 
There are therefore useful relationships with Ehrhart polynomials.   The study of reflexive polytopes led Batyrev to discover a duality between two different types of Calabi-Yau manifolds, and thus found further applications in mathematical physics.



\bigskip
\begin{thebibliography}{100}



\bibitem{Barany}
Imre B\'ar\'any, \emph{Random points and lattice points in convex bodies}, Bull. Amer. Math. Soc. (N.S.) \textbf{45} (2008), no. 3, 339--365.


\bibitem{barvinok1}
Alexander Barvinok, \emph{Exponential integrals and sums over convex
polyhedra}, Funktsional. Anal. i Prilozhen. \textbf{26} (1992), no.~2,
64--66.

\bibitem{barvinok2}
\bysame,  \emph{A polynomial time algorithm for counting integral
points in polyhedra when the dimension is fixed}, Math. Oper. Res.
\textbf{19} (1994), no.~4, 769--779.

\bibitem{barvinok3}
\bysame, \emph{Integer points in polyhedra}, Zurich Lectures in Advanced
Mathematics, European Mathematical Society (EMS), Zurich, 2008.

\bibitem{barvinokpommersheim}
Alexander Barvinok and James~E. Pommersheim, \emph{An algorithmic theory of
lattice points in polyhedra}, New {P}erspectives in {A}lgebraic
{C}ombinatorics (Berkeley, CA, 1996--97), Math. Sci. Res. Inst. Publ.,
vol.~38, Cambridge Univ. Press, Cambridge, 1999, pp.~91--147.



\bibitem{Batyrev}
Victor~V. Batyrev, \emph{Dual polyhedra and mirror symmetry for Calabi--Yau
hypersurfaces in toric varieties}, J. Algebraic Geom. \textbf{3} (1994),
no.~3, 493--535, {\tt arXiv:alg-geom/9310003}.

\bibitem{JozsefBeck}
J\'ozsef Beck, \emph{Probabilistic Diophantine approximation, Randomness in lattice point counting}, Springer Monographs in Mathematics, Springer, Cham, (2014), 1--487.


\bibitem{BeckBraunKoppeSavageZafeirakopoulos}
Matthias Beck, Ben Braun, Matthias K\"oppe, Carla Savage, and Zafeirakis Zafeirakopoulos, \emph{s-Lecture hall partitions, self-reciprocal polynomials, and Gorenstein cones}, Ramanujan Journal \textbf{36} (2015), 123--147.



\bibitem{BeckRobinsSam}
Matthias Beck, Sinai Robins, and Steven~V Sam, \emph{Positivity theorems for
solid-angle polynomials}, Beitr\"age Algebra Geom. \textbf{51} (2010), no.~2,
493--507, {\tt arXiv:0906.4031}.


\bibitem{BeckRobins}
Matthias Beck and Sinai Robins,
\emph{Computing the continuous discretely: integer-point enumeration in polyhedra}, 
$2$'nd edition, Springer, New York, (2015), 1--285.

\bibitem{BeckDiazRobins}
Matthias Beck, Ricardo Diaz, and Sinai Robins, \emph{The {F}robenius problem,
rational polytopes, and {F}ourier--{D}edekind sums}, J. Number Theory
\textbf{96} (2002), no.~1, 1--21, {\tt arXiv:math.NT/0204035}.

\bibitem{BeckRobinsPolygons}
Matthias Beck and Sinai Robins, \emph{Explicit and efficient formulas for the
lattice point count in rational polygons using {D}edekind--{R}ademacher
sums}, Discrete Comput. Geom. \textbf{27} (2002), no.~4, 443--459, {\tt
arXiv:math.CO/0111329}.

\bibitem{BeckSamWoods}
Matthias Beck,  Steven Sam, and Kevin Woods, 
\emph{Maximal periods of Ehrhart quasi-polynomials},
J. Combin. Theory Ser. A \textbf{115} (2008), no. 3, 517--525. 

\bibitem{brionvergne}
Michel Brion and Mich{\`e}le Vergne, \emph{Residue formulae, vector partition
functions and lattice points in rational polytopes}, J. Amer. Math. Soc.
\textbf{10} (1997), no.~4, 797--833.

\bibitem{camengasolidangles}
Kristin~A. Camenga, \emph{Vector spaces spanned by the angle sums of polytopes}, Beitr\"age Algebra Geom. \textbf{47} (2006), no.~2, 447--462,
{\tt arXiv:math.MG/0508629}.


\bibitem{CLS}
Dan Cristofaro-Gardiner, Teresa Xueshan Li,  and Richard Stanley,
\emph{New examples of period collapse}, (2015).
arXiv:1509.01887v1

\bibitem{danilov}
Vladimir~I. Danilov, \emph{The geometry of toric varieties}, Uspekhi Mat. Nauk
\textbf{33} (1978), 85--134, 247.

\bibitem{deloerahemmeckekoeppe}
Jes{\'u}s~A. De~Loera, Raymond Hemmecke, and Matthias K{\"o}ppe,
\emph{Algebraic and {G}eometric {I}deas in the {T}heory of {D}iscrete
{O}ptimization}, MOS-SIAM Series on Optimization, vol.~14, Society for
Industrial and Applied Mathematics (SIAM), Philadelphia, PA; Mathematical
Optimization Society, Philadelphia, PA, 2013.

\bibitem{desariorobins}
David Desario and Sinai Robins, \emph{Generalized solid-angle theory for real
polytopes}, The Quarterly Journal of Mathematics, \textbf{62} (2011), no.~4, 1003--1015, {\tt
arXiv:0708.0042}.

\bibitem{diaconisgangoli}
Persi Diaconis and Anil Gangolli, \emph{Rectangular arrays with fixed margins},
Discrete Probability and Algorithms (Minneapolis, MN, 1993), Springer, New
York, 1995, pp.~15--41.

\bibitem{diazrobins}
Ricardo Diaz and Sinai Robins, \emph{The {E}hrhart polynomial of a lattice polytope}, Ann. of Math.
(2) \textbf{145} (1997), no.~3, 503--518.

\bibitem{Ehrhart1}
Eug{\`e}ne Ehrhart, \emph{Sur les poly\`edres rationnels homoth\'etiques \`a
{$n$}\ dimensions}, C. R. Acad. Sci. Paris \textbf{254} (1962), 616--618.

\bibitem{Ehrhart2}
\bysame,  \emph{Sur un probl\`{e}me de g\'{e}om\'{e}trie diophantienne lin\'{e}aire II},  J. reine. angew. Math. 227, (1967), 25--49.

\bibitem{ehrhartbook}
\bysame, \emph{Polyn\^omes arithm\'etiques et m\'ethode des poly\`edres en
combinatoire}, Birkh\"auser Verlag, Basel, 1977, International Series of
Numerical Mathematics, Vol. 35.
 
\bibitem{fukshanskyrobinssolidangle}
Lenny Fukshansky and Sinai Robins, \emph{Bounds for solid angles of lattices of
rank three}, J. Combin. Theory Ser. A \textbf{118} (2011), no.~2, 690--701,
{\tt arXiv:1006.0743}.

\bibitem{fulton}
William Fulton, \emph{Introduction to {T}oric {V}arieties}, Annals of
Mathematics Studies, vol. 131, Princeton University Press, Princeton, NJ,
1993.


\bibitem{gravinrobinsshiryaev}
Nick Gravin, Sinai Robins, and Dmitry Shiryaev, \emph{Translational tilings by
a polytope, with multiplicity}, Combinatorica \textbf{32} (2012), no.~6,
629--649, {\tt arXiv:1103.3163}.

\bibitem{haasemcallister}
Christian Haase and Tyrrell McAllister, \emph{Quasi-period collapse and GLn (Z)-scissors congruence in rational polytopes}, Contemp. Math. \textbf{452} (2008), 115-122.

\bibitem{HardyLittlewood}   Godfrey Harold Hardy and John Edensor Littlewood,  \emph{Some problems of Diophantine approximation: The lattice-points of a right-angled triangle (Second memoir)},   Abh. Math. Sem. Univ. Hamburg, 
no. 1 (1922), no. 1, 211--248.


\bibitem{henkschurmannroots}
Martin Henk, Achill Sch{\"u}rmann, and J{\"o}rg~M. Wills, \emph{Ehrhart
polynomials and successive minima}, Mathematika \textbf{52} (2005), no.~1--2,
1--16 (2006), {\tt arXiv:math.MG/0507528}.

\bibitem{kannan}
Ravi Kannan, \emph{Lattice translates of a polytope and the {F}robenius
problem}, Combinatorica \textbf{12} (1992), no.~2, 161--177.


\bibitem{Katz.Stapledon}
Eric Katz and Alan Stapledon, \emph{Local h-polynomials, invariants of subdivisions, and mixed Ehrhart theory}, Adv. Math., \textbf{286} (2016), 181--239.

\bibitem{Knuth}
Donald Knuth,  \emph{Notes on generalized Dedekind sums},  Acta Arith. (1977), 297--325.

\bibitem{lawrence}
Jim Lawrence, \emph{Polytope volume computation}, Math. Comp. \textbf{57} (1991),
no.~195, 259--271.

\bibitem{linke}
Eva Linke, \emph{Rational {E}hrhart quasi-polynomials}, J. Combin. Theory Ser.
A \textbf{118} (2011), no.~7, 1966--1978, {\tt arXiv:1006.5612}.

\bibitem{macdonald1}
Ian~G. Macdonald,
\emph{The volume of a lattice polyhedron},  Proc. Cambridge Philos. Soc., \textbf{59} (1963), 719--726.

\bibitem{macdonald2}
\bysame, \emph{Polynomials associated with finite cell-complexes}, J. London Math. Soc. (2) \textbf{4} (1971), 181--192.

\bibitem{mcallisterwoods}
Tyrrell B. McAllister and Kevin M. Woods,
\emph{The minimum period of the Ehrhart quasi-polynomial of a rational polytope},
J. Combin. Theory Ser. A \textbf{109} (2005), no.~2, 345-352.

\bibitem{McMullen1}
Peter McMullen,
\emph{Lattice invariant valuations on rational polytopes},  Arch. Math., 31, (1978), 509--516.

\bibitem{McMullen2}
\bysame, 
\emph{Non-linear angle-sum relations for polyhedral cones and polytopes},   Math. Proc. Cambridge Phil. Soc., 78, (1975), 247--261.

\bibitem{McMullen3} 
\bysame, \emph{Angle-sum relations for polyhedral sets},  Mathematika \textbf{33} (1986), no. 2, 173--188.

\bibitem{morelli}
Robert Morelli, \emph{Pick's theorem and the {T}odd class of a toric variety},
Adv. Math. \textbf{100} (1993), no.~2, 183--231.

\bibitem{payne}
Sam Payne, \emph{Ehrhart series and lattice triangulations}, Discrete Comput.
Geom. \textbf{40} (2008), no.~3, 365--376, {\tt arXiv:math/0702052}.

\bibitem{perles}
Micha~A. Perles and Geoffrey~C. Shephard, \emph{Angle sums of convex
polytopes}, Math. Scand. \textbf{21} (1967), 199--218.

\bibitem{pommersheim}
James~E. Pommersheim, \emph{Toric varieties, lattice points and {D}edekind
sums}, Math. Ann. \textbf{295} (1993), no.~1, 1--24.

\bibitem{postnikovpermutahedra}
Alexander Postnikov, \emph{Permutohedra, associahedra, and beyond}, Int. Math.
Res. Not. (2009), no.~6, 1026--1106, {\tt arXiv:math/0507163}.

\bibitem{Nhat.Sinai}
Quang-Nhat Le and Sinai Robins,  \emph{Macdonald's solid-angle sum for real dilations of rational polygons}, preprint.

\bibitem{Randol1}
Marina Nechayeva and Burton Randol,  \emph{Asymptotics of weighted lattice point counts inside dilating polygons}, Additive number theory,  Springer, New York, (2010), 287--301.

\bibitem{Randol2}
Burton Randol, \emph{On the number of integral lattice-points in dilations of algebraic polyhedra}, Internat. Math. Res. Notices (1997) no. 6, 259--270.

\bibitem{ribando}
Jason~M. Ribando, \emph{Measuring solid angles beyond dimension three},
Discrete Comput. Geom. \textbf{36} (2006), no.~3, 479--487.



\bibitem{schlafli}
Ludwig Schl\"afli, \emph{Theorie der vielfachen {K}ontinuit\"at}, Ludwig
{S}chl\"afli, 1814--1895, {G}esammelte {M}athematische {A}bhandlungen,
{V}ol.~I, Birkh\"auser, Basel, 1950, pp.~167--387.

\bibitem{schrijver}
Alexander Schrijver, \emph{Combinatorial {O}ptimization. {P}olyhedra and
{E}fficiency. {V}ol. {A}--{C}}, Algorithms and Combinatorics, vol.~24,
Springer-Verlag, Berlin, 2003.

\bibitem{shallit}
Jeffrey Shallit, \emph{The {F}robenius problem and its generalizations},
Developments in language theory, Lecture Notes in Comput. Sci., vol. 5257,
Springer, Berlin, 2008, pp.~72--83.

\bibitem{shephard}
Geoffrey~C. Shephard, \emph{An elementary proof of {G}ram's theorem for convex
polytopes}, Canad. J. Math. \textbf{19} (1967), 1214--1217.

\bibitem{siegel}
Carl~Ludwig Siegel, \emph{Lectures on the {G}eometry of {N}umbers},
Springer-Verlag, Berlin, 1989, Notes by B. Friedman, rewritten by Komaravolu
Chandrasekharan with the assistance of Rudolf Suter, with a preface by
Chandrasekharan.

\bibitem{Skriganov1}
Maxim~M. Skriganov, 
\emph{Ergodic theory on homogeneous spaces and the enumeration of lattice points in
polyhedra (Russian)}, Dokl. Akad. Nauk \textbf{355} (1997), no. 5, 609--611.

\bibitem{Skriganov2}
\bysame, \emph{Ergodic theory on SL(n), Diophantine approximations and anomalies in the lattice point problem}, Invent. Math. \textbf{132} (1998), no. 1, 1--72.

\bibitem{Skriganov3}
\bysame, \emph{On logarithmically small errors in the lattice point problem. (English summary)} Ergodic Theory Dynam. Systems \textbf{20} (2000), no. 5, 1469--1476.

\bibitem{sommerville}
Duncan M.~Y. Sommerville, \emph{The relation connecting the angle-sums and
volume of a polytope in space of $n$ dimensions}, Proc. Roy. Soc. London,
Ser. A \textbf{115} (1927), 103--119.

\bibitem{Stanley1} 
Richard~P. Stanley, \emph{Subdivisions and local h-vectors}, J. Amer. Math. Soc. \textbf{5} (1992), no. 4, 805--851.


\bibitem{Stanleyreciprocity}
\bysame,   \emph{Combinatorial reciprocity theorems}, Advances in Math.
\textbf{14} (1974), 194--253.

\bibitem{Stanleybook}
\bysame, \emph{Enumerative {C}ombinatorics. {V}olume 1}, Second ed., Cambridge
Studies in Advanced Mathematics, vol.~49, Cambridge University Press,
Cambridge, 2012.

\bibitem{stapledonadditive}
Alan Stapledon, \emph{Additive number theory and inequalities in {E}hrhart
theory}, Preprint ({\tt arXiv:0904.3035v2}).

\bibitem{stapledonweightedehrart}
\bysame, \emph{Weighted {E}hrhart theory and orbifold cohomology}, Adv. Math.
\textbf{219} (2008), no.~1, 63--88, {\tt arXiv:math/0711.4382}.

\bibitem{YauZhang}
Stephen T. Yau and Letian Zhang, \emph{An upper estimate of integral points in real simplices with an application to singularity theory}, Math. Res. Lett. \textbf{13} (2006), no. 6, 911--921.


\end{thebibliography}
\end{document}